\documentclass[11pt,a4paper]{article}
\usepackage{amsfonts,amsgen,amsmath,amstext,amsbsy,amsopn,amsthm,amsfonts,amssymb,amscd}
\usepackage{epsf,epsfig}
\usepackage{float}
\usepackage{ebezier,eepic}
\usepackage{color}
\usepackage{tikz}
\usepackage{multirow}
\usepackage{graphicx}
\usepackage{subfigure}
\newtheorem{thm}{Theorem}[section]

\newtheorem{prop}[thm]{Proposition}
\newtheorem{lem}[thm]{Lemma}
\newtheorem{false statement}{False statement}

\theoremstyle{definition}
\newtheorem{defn}[thm]{Definition}
\newtheorem{claim}{Claim}

\newtheorem{conj}[thm]{Conjecture}

\makeatletter \@addtoreset{equation}{section}

\title{\bf\Large A Stability Theorem for Maximal $C_{2k+1}$-free Graphs}

\author{
Jian Wang\thanks{Department of Mathematics,
Taiyuan University of Technology, Taiyuan 030024, P.~R.~China. E-mail: {\tt  wangjian01@tyut.edu.cn}. Research supported by NSFC No.11701407 and  Shanxi Province Science Foundation for Youths No. 201801D221028.}~~~~
Shipeng Wang\thanks{Department of Mathematics, Jiangsu University, Zhenjiang, Jiangsu 212013, P.~R. China. E-mail: {\tt  spwang22@yahoo.com}. Research supported by NSFC No.12001242.}~~~~
Weihua Yang\thanks{Department of Mathematics, Taiyuan University of Technology,Taiyuan 030024, P.~R. China. E-mail: {\tt  yangweihua@tyut.edu.cn}. Research supported by NSFC No.11671296.}~~~~
Xiaoli Yuan\thanks{Department of Mathematics,
Taiyuan University of Technology, Taiyuan 030024, P.~R.~China. E-mail: {\tt  tyutxiaoli@126.com}}~~~~
}

\begin{document}

\date{}

\maketitle

\begin{abstract}
For any positive integer $k$, we show that every maximal $C_{2k+1}$-free graph with at least $n^2/4-o(n^{3/2})$ edges contains an induced complete bipartite subgraph on $(1-o(1))n$ vertices. We also show that this is best possible.
\end{abstract}

\medskip

\section{Introduction}

Let $H$ be a graph. A graph is called {\it $H$-free} if it does not contain $H$ as a subgraph.  We use $ex(n,H)$ to denote the maximum number of edges in an $H$-free graph on $n$ vertices.
Let $T_r(n)$ denote the {\it Tur\'{a}n graph}, the complete $r$-partite graph on $n$ vertices with $r$ partition classes of size $\lfloor \frac{n}{r}\rfloor$ or
$\lceil \frac{n}{r}\rceil$, and put $t_r(n)=e(T_r(n))$. The classic Tur\'{a}n theorem~\cite{turan 1941} tells us that $T_r(n)$ is the unique graph attaining the maximum number of edges in $K_{r+1}$-free graphs.
Erd\H{o}s and Simonovits \cite{erdos1966,erdos1968,sim1974} discovered a stability  phenomenon  on Tur\'{a}n theorem: if $G$ is a $K_{r+1}$-free graph
with $t_r(n)-o(n^2)$ edges, then $G$  can be made into the Tur\'{a}n graph by adding and deleting $o(n^2)$ edges.  This kind of results have been  extensively studied in various discrete structure  (e.g.\cite{alon2004,sam2014,sw2009}).

An alternative form of stability for  Tur\'{a}n theorem is due to Brouwer \cite{brouwer1981},
who showed that if $n\geq 2r+1$ and $G$ is a $K_{r+1}$-free graph with $e(G)\geq t_r(n)-\lfloor \frac{n}{r}\rfloor +2$, then $G$ is $r$-partite. A graph $G$ is called {\it maximal $H$-free} if it is $H$-free and
the addition of any edge in the complement $\overline{G}$ creates a copy of $H$. Tyomkyn and Uzzel \cite{tyomkyn2015} considered an interesting question: when can one guarantee an
`almost spanning' complete $r$-partite subgraph in a maximal $K_{r+1}$-free graph $G$ with $t_r(n)-o(n^2)$ edges?
They showed that every maximal $K_{4}$-free graph $G$ with $t_3(n)-cn$ edges contains a complete $3$-partite subgraph on $(1-o(1))n$ vertices. Popielarz, Sahasrabudhe and Snyder completely answered this question and gave a tight result.

\begin{thm}[Popielarz, Sahasrabudhe and Snyder \cite{popie2018}]\label{PSSthm}
	Let $r\geq 2$ be an integer. Every maximal $K_{r+1}$-free on $n$ vertices with at least $t_r(n)-o(n^{\frac{r+1}{r}})$ edges contains an induced complete $r$-partite subgraph on $(1-o(1))n$ vertices.
\end{thm}

We continue  to consider this type of stability on maximal $C_{2k+1}$-free graphs. F\"{u}redi and Gunderson \cite{furedi2015} showed that the bipartite Tur\'{a}n graph attains the maximum number of edges in $C_{2k+1}$-free graphs.

\begin{thm}[F\"{u}redi and Gunderson, \cite{furedi2015}]\label{c2k1bound}
	For $k\geq 2$ and $
	n\geq 4k\geq 8$,
	\[
	ex(n,C_{2k+1})=\left\lfloor \frac{n^2}{4} \right\rfloor,
	\]
	and the unique extremal graph is the bipartite
	Tur\'{a}n graph.
\end{thm}

Let $f_k(n,m)$ be the maximum value $t$ such that every maximal $C_{2k+1}$-free graph with at least $\frac{n^2}{4}-t$ edges contains an induced complete bipartite subgraph on $n-m$ vertices.
Theorem \ref{PSSthm} implies that $f_1(n,o(n))=o(n^{3/2})$. One may guess that the magnitude of $f_k(n,o(n))$ is depending on $k$. Surprisingly, we show that $f_k(n,o(n))=o(n^{3/2})$ for all $k\geq 1$.

\begin{thm}\label{thm-0}
For  an integer $k\geq 2$, let $G$ be a maximal $C_{2k+1}$-free graph on $n$ vertices with  $e(G)\geq \frac{n^2}{4}-o(n^{3/2})$. Then $G$ contains an induced complete bipartite subgraph on $(1-o(1))n$ vertices.
\end{thm}

The follow lemma shows that every $C_{2k+1}$-free graph on $n$ vertices with $\frac{n^2}{4}-o(n^2)$ edges contains a large induced bipartite graph, which is a key ingredient in the proof of Theorem \ref{thm-0}.

\begin{lem}\label{thm-2}
For $k\geq 2$, $0<\varepsilon<\frac{1}{2000k^2}$ and $n\geq (2k)^{9k^2}$, let $G$ be a $C_{2k+1}$-free graph with  $e(G)\geq \frac{n^2}{4}-\varepsilon n^2$. Then there exists a $T\subset V(G)$ with $|T|\leq 50k \varepsilon n$ such that $G-T$ is a bipartite graph with at least $\frac{n^2}{4}-25k\varepsilon n^2$ edges and minimum degree at least $(\frac{1}{2}-\frac{1}{16k})n$.
\end{lem}

It should be mentioned that Kor\'{a}ndi, Roberts and Scott \cite{scott2020} proposed the following conjecture, which may lead to a tight upper bound on $|T|$ in Lemma \ref{thm-2} (if true). For a graph $G$, let $D_2(G)$  denote the minimum number of edges that must be
removed from $G$ to make it bipartite.

\begin{conj}[Kor\'{a}ndi, Roberts and Scott, \cite{scott2020}]
 Let $k \geq  2$, and suppose $G$ is a $C_{2k+1}$-free graph with $n$ vertices and at least
$\frac{n^2}{4}-\varepsilon n^2$ edges. Then some blowup $G^*$ of $C_{2k+3}$ satisfies $e(G^*) \geq e(G)$ and $D_2(G^*) \geq D_2(G)$.
\end{conj}

We also give a similar construction as in \cite{popie2018} to show that the magnitude of $f_k(n,o(n))$ in
Theorem \ref{thm-0} is tight. Precisely, for every $\delta>0$, there exists a maximal $C_{2k+1}$-free graph with at least $\frac{n^2}{4}-\delta n^{3/2}$ edges for which the conclusion of Theorem~\ref{thm-0} fails.

The rest of this paper is organized as follows. We finish the proof of Lemma \ref{thm-2} in the next section and  in Section 3, we prove our main result, Theorem \ref{thm-0}. In Section 4, we provide constructions to show that Theorem \ref{thm-0} is best possible.

We follow standard notation through.
Let $G$ be a graph. Denote by $\bar{G}$ the complement graph of $G$. For any $v\in V(G)$, we use $N_G(v)$ and $\deg_G(v)$ to denote the neighborhood and degree of $v$ in $G$, respectively.  Let $S$  be a subset of $V(G)$. We use $\deg_G(v,S)$ to denote the number of neighbors of $v$ in $S$. Denote by $G-S$ the subgraph of $G$ induced by $V(G)\setminus S$. When $S = \{v\}$, we simply write $G-v$.  Denote by $e_G(S)$ the number of edges of $G$ in $S$. For any two disjoint subsets $X,Y$ of $V(G)$, let $G[X,Y]$ denote the subgraph of $G$ induced by the edge set $$\{xy\in E(G)\colon x\in X \mbox{ and } y\in Y\}.$$
Denote by $e_G(X,Y)$  the number of edges in $G[X,Y]$. We often omit the subscript when the underline graph is clear.

\section{Finding large bipartite subgraphs in $C_{2k+1}$-free graphs}

In this section, we shall prove Lemma \ref{thm-2}, which is a key ingredient to the proof of Theorem \ref{thm-0}. The following two results are needed.

\begin{thm}[Bukh and Jiang, \cite{bukh-jiang}]\label{c2kbound}
For $n\geq (2k)^{8k^2}$,
\[
ex(n,C_{2k}) \leq 80\sqrt{k} \log k\cdot n^{1+1/k} +10k^2n.
\]
\end{thm}

Let $P_{k+1}$ denote a path of length $k$.

\begin{thm}[Erd\H{o}s and Gallai, \cite{erdos-gallai}]\label{pathbound}
For $k\geq 1$,
\[
ex(n,P_{k+1}) \leq \frac{(k-1)n}{2}.
\]
\end{thm}

Now we prove the following lemma, which will be used in the proof of Lemma \ref{thm-2}.

\begin{lem}\label{lem-1}
Let $k\geq 2,n\geq (2k)^{8k^2}$ be integers,  and  let $\alpha,\beta$ be real numbers with
$(2k+1)/(8k+6)<\alpha \leq 1/4$ and $1/2-2\alpha\leq \beta\leq \alpha/(2k+1)$. If $G$ is a $C_{2k+1}$-free graph on $n$ vertices with $\alpha n^2$ edges  and $\delta(G) \geq (1/2-\beta)n$, then $G$ is bipartite.
\end{lem}

\begin{proof}
By Theorem \ref{c2kbound}, we have
\[
e(G)>100\sqrt{k} \log k\cdot n^{1+1/k}>ex(n,C_{2k})
 \]
 for $n\geq (2k)^{8k^2}$. It follows that $G$ contains a cycle of length $2k$. Let  $C=v_1v_2\ldots v_{2k}v_1$ be such a cycle and $U=V(G)\setminus V(C)$. Since $G$ is $C_{2k+1}$-free, every vertex in $U$ has at most $k$ neighbors in $V(C)$. Then $U$ can be partitioned into three classes.
\begin{equation*}
\left\{
\begin{array}{l}S_{odd} =\{u\in U\colon uv_{2i-1}\in E(G),\ \mbox{for } i=1,2\ldots,k\};\\[6pt]
S_{even}=\{u\in U\colon uv_{2i}\in E(G),\ \mbox{for } i=1,2\ldots,k\};\\[6pt]
S=S_{odd}\cup S_{even},\ S'=U\setminus S.
\end{array}\right.
\end{equation*}

By the definition of $S$ and $S'$, we have
\begin{align}\label{eq-2-1}
e(V(C),U) &\leq  k|S|+(k-1)|S'|\nonumber\\
 &= k|S| +(k-1)(n-2k-|S|)\nonumber \\
 &=|S|+(k-1)(n-2k).
\end{align}
On the other hand, every vertex in $V(C)$ has at least $\delta(G)-(2k-1)$ neighbors in $U$. It follows that
\begin{align}\label{eq-2-2}
e(V(C),U) &= \sum_{i=1}^{2k} \deg(v_i,U)\nonumber\\
 &\geq \sum_{i=1}^{2k} (\delta(G)-2k+1)\nonumber \\
 &\geq 2k(1/2-\beta)n-2k(2k-1).
\end{align}

By \eqref{eq-2-1} and \eqref{eq-2-2}, we obtain that
\begin{align}\label{ineq-3}
|S| \geq (1-2k\beta)n -2k^2
\end{align}
and
\begin{align}\label{ineq-4}
|S'| = n-|V(C)|-|S| \leq 2k\beta n+2k^2-2k.
\end{align}

We claim that $S_{odd}$ and $S_{even}$ are both independent sets in $G$. Otherwise, we may assume, without loss of generality, that $S_{odd}$  is not an independent set. Then there is an edge $xy$ in $G[S_{odd}]$, but then $xyv_1v_2\ldots v_{2k-1}x$ is a cycle of length $2k+1$, a contradiction. We now show that $S_{odd}\neq \emptyset$ and $S_{even}\neq \emptyset$. If not, without loss of generality, we may assume that $S_{odd}= \emptyset$.
 Then for any $x\in S_{even}$, by \eqref{ineq-3} we have
\[
\deg(x)  \leq n-|S| \leq 2k\beta n + 2k^2\leq \alpha n+2k^2-\beta n\leq \frac{1}{4} n +2k^2-\beta n.
\]
Since $n> 8k^2$, we obtain that
\[
\deg(x)< (1/2-\beta)n \leq \delta(G),
\]
a contradiction.

For any vertex $x\in S'$,  we show that either $\deg(x,S_{odd})=0$ or $\deg(x,S_{even})=0$ holds. Suppose not, let $x$ be a vertex in $S'$ such that $x$ has two neighbors $y\in S_{odd}$ and $z\in S_{even}$. Then $xyv_1v_2\ldots v_{2k-2}zx$ is a cycle of length $2k+1$,  a contradiction. Let $S_{odd}'$, $S_{even}'$ be a partition of $S'$ such that
\[
S_{odd}' =\{u\in S'\colon \deg(u,S_{odd})=0\} \mbox{ and } S_{even}'=\{u\in S'\colon \deg(u,S_{even})=0\}.
\]

\begin{claim}
$S_{odd}\cup S_{odd}'$ and $S_{even}\cup S_{even}'$ are both independent sets of $G$.
\end{claim}

\begin{proof}
By contradiction, without loss of generality, we may assume that $S_{odd}\cup S_{odd}'$ is not an independent set of $G$, then there is an edge $w_1w_2$ in $G[S_{odd}\cup S_{odd}']$.  Note that $S_{odd}$ is an independent set. By the definition of $S_{odd}'$, any vertex in $S_{odd}'$ has no neighbor in $S_{odd}$, implying that $w_1,w_2\in S_{odd}'$.
 For each $i\in\{1,2\}$,
$\deg(w_i,S_{odd})=0$ and by \eqref{ineq-4}
\begin{align*}
\deg(w_i,S_{even})
&\geq \delta(G) - |S'|-\deg(w_i,V(C))\\
&\geq \left(\frac{1}{2}-\beta\right)n -(2k\beta n+2k^2-2k)-(k-1)\\
&\geq  \left(\frac{1}{2}-\alpha\right)n -2k^2+k+1\\
&\geq 2.
\end{align*}
 Therefore, there are two distinct vertices $x$ and $y$ in $S_{even}$ such that $xw_1,yw_2\in E(G)$. But then $w_1xv_2v_3\ldots v_{2k-2}yw_2w_1$ is a cycle of length $2k+1$, a contradiction. This proves the claim.
\end{proof}

\begin{claim}
$\{v_1, v_3,\ldots, v_{2k-1}\}$ and $\{v_2, v_4,\ldots, v_{2k}\}$ are both independent sets of $G$.
\end{claim}

\begin{proof}
By contradiction, without loss of generality, we may assume that $\{v_1, v_3,\ldots, v_{2k-1}\}$ is not an independent set of $G$. Without loss of generality, let $v_1v_{2i+1}$ be an edge for some $i$.
 Note that  $S_{odd}\cup S_{even}$ induces a bipartite graph.
If there is a path $P$ of length $2k-3$ in $G[S_{odd},S_{even}]$, say $P=u_1u_2\ldots u_{2k-2}$ with $u_1\in S_{odd}$ and $u_{2k-2}\in S_{even}$, then $v_{2i}v_{2i+1}v_1u_1u_2\ldots u_{2k-2}v_{2i}$ is a cycle of length $2k+1$, a contraction. Thus $G[S_{odd},S_{even}]$ is $P_{2k-2}$-free.  Then, by Theorem \ref{pathbound} we have
\begin{align}\label{ineq-1}
e(S)  \leq \frac{(2k-3)|S|}{2}\leq \frac{(2k-3)n}{2}.
\end{align}
Furthermore,
\begin{align}\label{ineq-2}
e(G) &\leq  e(S) +e(S,S'\cup V(C))+2e(S'\cup V(C))= e(S)+\sum_{v\in S'\cup V(C)} \deg(v).
\end{align}
Substituting \eqref{ineq-1} to \eqref{ineq-2}, we obtain that
\begin{align*}
e(G) \leq \frac{(2k-3)n}{2}+\sum_{v\in S'\cup V(C)} \deg(v)\leq \frac{(2k-3)n}{2}+n(|S'|+|V(C)|).
\end{align*}
By \eqref{ineq-4} and $\beta\leq\frac{\alpha}{2k+1}$, we further have
\begin{align*}
e(G) &\leq  \frac{(2k-3)n}{2}+n(2k\beta n+2k^2-2k+2k)\\
&= 2k\beta n^2+\left(2k^2+\frac{2k-3}{2}\right)n\\
&\leq \frac{2k}{2k+1}\alpha n^2+\left(2k^2+k\right)n.
\end{align*}
Since $n> \frac{k(2k+1)^2}{\alpha}$, we obtain that $e(G)<\alpha n^2$,
a contradiction. This proves the claim.
\end{proof}

For any $x\in S_{odd}'$, $x$ cannot be adjacent to any vertex in $\{v_2,v_{4},\ldots,v_{2k}\}$. Otherwise, we may assume, without loss of generality, that $xv_2\in E(G)$. Similar as in the proof of Claim 1, we have
 \begin{align*}
\deg(x,S_{even})\geq \delta(G) - |S'|-\deg(x,V(C))\geq 1.
\end{align*}
Thus there is a vertex $y$ in $S_{even}$ such that $xy\in E(G)$, but then $xv_2v_3\ldots v_{2k}yx$ is a cycle of length $2k+1$, a contradiction. Hence $N(x)\cap V(C)\subset\{v_1, v_3,\ldots, v_{2k-1}\}$ for any $x\in S_{odd}'$. Similarly, $N(x)\cap V(C)\subset\{v_2, v_4,\ldots, v_{2k}\}$ for any $x\in S_{even}'$. By Claims 1 and 2, we conclude that $S_{odd}\cup S_{odd}'\cup \{v_2, v_4,\ldots, v_{2k}\}$ and $S_{even}\cup S_{even}'\cup \{v_1, v_3,\ldots, v_{2k-1}\}$ are both independent sets of $G$. Thus $G$ is bipartite.
\end{proof}

Now we apply Lemma  2.3 to prove Lemma \ref{thm-2}.

\begin{proof}[Proof of Lemma \ref{thm-2}]
We begin by finding a subgraph $G'$ of $G$ with large minimum degree. Set $G_0=G$ and do this by deleting vertices one at a time, forming graphs $G_0,G_1,\ldots,G_i,\ldots$ In the $i$-th step, if there is a vertex $v$ with degree less than $(\frac{1}{2}-\frac{1}{20k})(n-i)$ in $G_i$, then let $G_{i+1}=G_{i}-v$. Otherwise, we stop and let $G'=G_i$ and $n'=|V(G')|$. If $n'\leq 4k$, then
\begin{align*}
e(G) &\leq e(G') +\sum_{i=0}^{n-n'-1} \left(\frac{1}{2}-\frac{1}{20k}\right) (n-i)\\
&\leq \binom{4k}{2}+ \left(\frac{1}{2}-
	\frac{1}{20k}
\right)\frac{n^2+n}{2}\\
&<\frac{n^2}{4}-\frac{n^2}{80k}+ \left(8k^2-\frac{n}{40k}\right) + \left(\frac{n}{4}-\frac{n^2}{80k}\right).
\end{align*}
Since $n\geq 320k^3$, it follows that
\[
e(G) \leq \frac{n^2}{4}-\frac{n^2}{80k} < \frac{n^2}{4}-\varepsilon n^2,
\]
which contradicts the assumption that $e(G)\geq \frac{n^2}{4}-\varepsilon n^2$.

Hence $n'> 4k$. Since $G'$ is $C_{2k+1}$-free, by Theorem \ref{c2k1bound} we have $e(G')\leq \frac{n'^2}{4}$. Then,
\begin{align}\label{eq-2-3}
e(G) &\leq e(G') +\sum_{i=0}^{n-n'-1} \left(\frac{1}{2}-\frac{1}{20k}\right) (n-i)\nonumber\\
 & \leq \frac{n'^2}{4}+ \left(\frac{1}{2}-\frac{1}{20k}\right)\frac{(n+n'+1)(n-n')}{2}\nonumber\\
 & \leq \frac{n^2}{4}+ \frac{n-n'}{4} -\frac{n^2-n'^2}{40k}.
\end{align}
If $n'<n/2$, then
\[
\frac{n^2-n'^2}{40k}-\frac{n-n'}{4} > \frac{3n^2}{160k}-\frac{n}{4}\geq \frac{3n^2}{320k}> \varepsilon n^2.
\]
Together this with \eqref{eq-2-3},
we obtain $e(G)<n^2/4-\varepsilon n^2$,
a contradiction.

Hence $n'\geq n/2$. Let
\[
f(x) =\frac{x^2}{40k}-\frac{x}{4}.
\]
By Lagrange's Mean Value Theorem, there is a $\xi \in (n',n)$ such that
\begin{align*}
f(n)-f(n') =f'(\xi)(n-n')=\left(\frac{\xi}{20k}-\frac{1}{4}\right)(n-n').
\end{align*}
Since $\xi> n'\geq \frac{n}{2}$, it follows that
\begin{align}\label{eq-2-4}
f(n)-f(n') \geq \left(\frac{n}{40k}-\frac{1}{4}\right)(n-n')\geq \frac{n}{50k}(n-n').
\end{align}

 By \eqref{eq-2-3}, \eqref{eq-2-4}
 and $e(G)\geq \frac{n^2}{4}-\varepsilon n^2$, we obtain that $\frac{n}{50k}(n-n')\leq \varepsilon n^2$,
implying that $n'\geq (1-50k\varepsilon) n$. Let $T= V(G)\setminus V(G')$. Then $|T|\leq 50k\varepsilon n$. Moreover,
\begin{align*}
e(G')\geq e(G)-\left(\frac{1}{2}-\frac{1}{20k}\right)n|T|\geq \frac{n^2}{4}-25k\varepsilon n^2,
\end{align*}
and
\begin{align*}
\delta(G')&\geq \left(\frac{1}{2}-\frac{1}{20k}\right)n'\geq \left(\frac{1}{2}-\frac{1}{20k}-25k\varepsilon+\frac{5}{2}\varepsilon\right)n\geq \left(\frac{1}{2}-\frac{1}{16k}\right)n.
\end{align*}
Since $n\geq (2k)^{9k^2}$, we have
\[
n'\geq (1-50k\varepsilon)n \geq \left(1-\frac{1}{400k}\right) n\geq (2k)^{8k^2}.
\]
Apply Lemma \ref{lem-1} to $G'$ with $\alpha = 1/4-25k\varepsilon$ and $\beta=1/(16k)$, we conclude that $G'$ is bipartite. Thus the theorem holds.
\end{proof}

\section{Making the bipartite subgraph complete}

In this section, we prove a quantitative form of Theorem \ref{thm-0}.

\begin{thm}\label{thm-1}
For two integers  $k\geq 2,n\geq (2k)^{9k^2}$, and real number $\varepsilon \in (0,\frac{1}{2000k^2})$, let $G$ be a maximal $C_{2k+1}$-free graph on $n$ vertices with  $e(G)= n^2/4-\varepsilon n^{\frac{3}{2}}$. Then $G$ contains an induced complete bipartite subgraph on $(1-250k^2\varepsilon)n$ vertices.
\end{thm}

\begin{proof}
By Lemma \ref{thm-2}, there is a $T\subset V(G)$ with $|T|\leq 50k \varepsilon n$ such that $G-T$ is a bipartite graph with at least $n^2/4-25k\varepsilon n^2$ edges and minimum degree at least $\left(\frac{1}{2}-\frac{1}{16k}\right)n$. Let $X$ and $Y$ be two partite classes of $G-T$. Since
\[
\delta(G-T) \geq \left(\frac{1}{2}-\frac{1}{16k}\right)n,
\]
we have $|X|,|Y|\geq \left(\frac{1}{2}-\frac{1}{16k}\right)n$, implying that $|X|,|Y|\leq \left(\frac{1}{2}+\frac{1}{16k}\right)n$ because $|X|+|Y|\leq n$.

\begin{claim} \label{pathtype}
For any non-edge $xy$ in $G-T$ with $x\in X$ and $y\in Y$, there is a path $P$ of length $2k$ connecting $x$ and $y$ in $G$, say $P=xx'u_1\ldots u_{2k-3}y'y$, such that $x',y'\in T$.
\end{claim}

\begin{proof}
Since $G$ is  a maximal $C_{2k+1}$-free graph, $G+xy$ contains a cycle $C$ of length $2k+1$ and $xy\in E(C)$. Let $Q=xx'u_1\ldots u_{2k-3} y'y$ be a path of length $2k$ in $G$.  If both $x'$ and $y'$ are in $T$, then we are done. Hence we assume that $\{x',y'\}\not\subset T$. Without loss of generality, we may assume that $x'\notin T$, then $x'\in Y$ because $G-T$ is bipartite. Note that
\begin{align*}
 |N_{G-T}(x')\cap N_{G-T}(y)| &\geq \deg_{G-T}(x') +\deg_{G-T}(y)-|X|\\
 & \geq 2\left(\frac{1}{2}-\frac{1}{16k}\right)n -\left(\frac{1}{2}+\frac{1}{16k}\right)n\\
 & = \left(\frac{1}{2}-\frac{3}{16k}\right)n\\
 &> 2k.
\end{align*}
Thus there is a vertex $u\in (N_{G-T}(x')\cap N_{G-T}(y))\setminus V(Q)$, but then $ux'u_1\ldots u_{2k-3}y'yu$ is a cycle of length $2k+1$ in $G$, a contradiction. Hence $x',y'\in T$ and the claim follows.
\end{proof}

Now we delete vertices incident with non-edges in $G-T$ until the resulting graph is complete bipartite by a greedy algorithm. Let $G_0=G-T$ and do this by deleting a set of vertices at a time, forming graphs $G_0,G_1,\ldots,G_i,\ldots,G_{l}$.  In the $i$-th step, if $G_i$ is complete bipartite, then we stop and let $l=i$. If $G_i$ is not complete bipartite, then there is an edge $x_iy_i$ in $\bar{G_i}$ with $x_i\in X$ and $y_i\in Y$. By Claim \ref{pathtype}, there is a path $Q_i$ of length $2k$ connecting $x_i$ and $y_i$
in $G$, say
\[
Q_i=x_ix_i'u_{i,1}\ldots u_{i,2k-3} y_i'y_i,
\]
with $x_i',y_i'\in T$.
Let $X_i = N_{G_i}(x_i')\cap X$, $Y_i = N_{G_i}(y_i')\cap Y$, $S_i$ be the one of $X_i$ and $Y_i$ with smaller size and let $G_{i+1} =G_i-S_i$.

\begin{claim}\label{claim-4}
For each $i\in \{0,1,\ldots,l-1\}$, there are at least $\frac{|S_i|^2}{16k^2}$ non-edges between $X_i$ and $Y_i$.
\end{claim}

\begin{proof}
 Note that
\[
Q_i=x_ix_i'u_{i,1}\ldots u_{i,2k-3} y_i'y_i
 \]
is a path of length $2k$. For any edge $ab$ in $G$ with $a\in X_i$ and $b\in Y_i$, if $\{a,b\}\cap \{u_{i,1},\ldots, u_{i,2k-3}\}=\emptyset$, then $ax_i'u_{i,1}\ldots u_{i,2k-3}y_i'ba$ is a cycle of length $2k+1$, a contradiction.
Hence every edge between $X_i$ and $Y_i$
intersects $\{u_{i,1},\ldots, u_{i,2k-3}\}$. Since $u_{i,1}\ldots u_{i,2k-3}$ is a path and $G_i$ is bipartite, we have
\[
|\{u_{i,1},\ldots, u_{i,2k-3}\} \cap X_i|< k \mbox{ and } |\{u_{i,1},\ldots, u_{i,2k-3}\} \cap Y_i|< k,
\]
implying that
\[
e(X_i,Y_i)<k(|X_i|+|Y_i|).
\]
It follows that
\[
e_{\bar{G}}(X_i,Y_i) = |X_i||Y_i| - e(X_i,Y_i) >|X_i||Y_i|-k(|X_i|+|Y_i|)
\]

 Without loss of generality, we may assume that $|X_i|\leq |Y_i|$, then $S_i = X_i$. If $|S_i|\geq 4k$, then
\begin{align*}
e_{\bar{G}}(X_i,Y_i)& >|S_i||Y_i|-k(|S_i|+|Y_i|)\\
& = |Y_i|(|S_i|-k)-k|S_i|\\
& \geq |S_i|^2-2k|S_i|\\
& \geq   \frac{|S_i|^2}{2}.
\end{align*}
If $|S_i|< 4k$, then since $x_iy_i$ is a non-edge between $X_i$ and $Y_i$, we have
\[
e_{\bar{G}}(X_i,Y_i) \geq 1 >  \frac{|S_i|^2}{16k^2}.
\]
Thus the claim holds.
\end{proof}

Note that there are $\sum_{i=0}^{l-1} |S_i|$ vertices deleted from $G-T$ and $G_l$ is complete bipartite. It follows that
\[
e_{\bar{G}}(X,Y) = \sum_{i=0}^{l-1} e_{\bar{G}}(X_i,Y_i).
\]
By Claim \ref{claim-4}, we obtain that
\begin{align}\label{eq-3-1}
\sum_{i=0}^{l-1} \frac{|S_i|^2}{16k^2} &\leq  e_{\bar{G}}(X,Y) \nonumber\\
&\leq |X||Y| - e(G-T)\nonumber \\
&\leq \frac{n^2}{4} -\left(\frac{n^2}{4}-25k\varepsilon n^2\right)\nonumber \\
&\leq  25k\varepsilon n^2.
\end{align}
By Cauchy-Schwarz inequality, we have
\begin{align}\label{eq-3-2}
\left(\sum_{i=0}^{l-1} |S_i|\right)^2 \leq \left(\sum_{i=0}^{l-1} |S_i|^2\right)l
\end{align}
In each step of the greedy algorithm, there is a $u\in T$ such that either $N(u)\cap X$ or $N(u)\cap Y$ is deleted.
It follows that $l\leq 2|T|$.
By \eqref{eq-3-1} and \eqref{eq-3-2},
\begin{align*}
\left(\sum_{i=0}^{l-1} |S_i|\right)^2 \leq 16\cdot 25 k^3 \varepsilon n^2 \cdot 2|T| \leq 16\cdot 25\cdot 100 k^4 \varepsilon^2 n^3.
\end{align*}
Thus,
\[
\sum_{i=0}^{l-1} |S_i| \leq 200 k^2\varepsilon n^{\frac{3}{2}}.
\]
Therefore, the total number of vertices deleted is at most

\begin{align*}
\sum_{i=0}^{l-1} |S_i|+|T|& =200 k^2\varepsilon n^{\frac{3}{2}}+ 50k\varepsilon n \leq 250 k^2\varepsilon n^{\frac{3}{2}}
\end{align*}
This completes the proof.
\end{proof}

\section{The lower bound constructions}

\begin{defn}\label{defn-1}
Given  $k\geq 2$ and $0<\alpha<\frac{1}{2}$, we define a class of graphs $\mathcal{G}_{k,\alpha}(n)$. Let
$t =\lceil \sqrt{\frac{\alpha n}{4k}} \rceil$.  A graph $G$ on $n$ vertices is in $\mathcal{G}_{k,\alpha}(n)$ if $V(G)$ can be partitioned into subsets $X_1,\ldots,X_t,X_{t+1}, Y_1,\ldots,Y_t,Y_{t+1}$, $Z_1,\ldots,Z_t$ such that:

\begin{itemize}
  \item[(i)] For each $i=1,\ldots,t$, $|Z_i|=2k-1$ and $G[Z_i]$ contains a path of length $2k-2$, say $P^i=z^i_1z^i_2\ldots z^i_{2k-1}$.
  \item[(ii)] For each $i =1,\ldots,t$,
  \[
  |X_i|=|Y_i| =\left \lfloor \frac{\alpha n-(2k-1)t}{2t} \right\rfloor
  \]
  and $X_{t+1},Y_{t+1}$  is a balanced partition of
  \[
  V(G)\setminus \bigcup_{i=1}^t(X_i\cup Y_i\cup Z_i).
  \]
  \item[(iii)] Both $X_1\cup\cdots \cup X_t\cup X_{t+1}$ and $Y_1\cup\cdots \cup Y_t \cup Y_{t+1}$ are independent sets in $G$.
  \item[(iv)] For each $i=1,\ldots,t$, $G[X_i,Y_i]$ is empty and $G[X_{t+1},Y_{t+1}]$ is complete. For each $i,j\in \{1,\ldots ,t+1\}$ with $i\neq j$, $G[X_i,Y_j]$ is complete.
  \item[(v)] For each $i= 1,\ldots,t$, $z_1^i$ is adjacent to every vertex in $X_i$ and  $z_{2k-1}^i$ is adjacent to every vertex in $Y_i$.
\end{itemize}
\end{defn}

When refer to vertex classes of a graph in $\mathcal{G}_{k,\alpha}$, we use $X,Y$ and $Z$ to denote $X_1\cup\cdots \cup X_{t+1}$, $Y_1\cup\cdots \cup Y_{t+1}$ and  $Z_1\cup\cdots \cup Z_t$, respectively.

\begin{prop}
$\mathcal{G}_{k,\alpha}(n)$ contains a $C_{2k+1}$-free graph.
\end{prop}

\begin{proof}
Let $G$ be a graph in $\mathcal{G}_{k,\alpha}(n)$ with minimum number of edges. By Definition \ref{defn-1} (i) and (v),  $N(z_1^i)\cap (X\cup Y)=X_i$, $N(z_{2k-1}^i)\cap (X\cup Y)=Y_i$ and $G[Z_i]$ is a path of length $2k-2$ for $i=1,\ldots,t$.
We shall show that $G$ is  $C_{2k+1}$-free. Suppose not, let $C$ be a cycle of length $2k+1$ in $G$. By Definition \ref{defn-1} (iii), $G[X,Y]$ is bipartite, implying that $V(C)\cap Z\neq \emptyset$. Since $Z$ induces $t$ vertex-disjoint paths of length $2k-2$ each, there exists an $i\in\{1,\ldots,t\}$ such that $z^i_1z^i_2\ldots z^i_{2k-1}$  is a segment of $C$, it follows that there is a path of length 3 between $z^i_1$ and $z^i_{2k-1}$ in $G$. However, $N(z_1^i)\cap (X\cup Y)=X_i$, $N(z_{2k-1}^i)\cap (X\cup Y)=Y_i$ and $G[X_i,Y_i]$ is empty, which leads to a contradiction. Thus, $G$ is $C_{2k+1}$-free.
\end{proof}

The following proposition shows that Theorem~\ref{thm-0} is tight.
\begin{prop}
Let $0<\alpha<\frac{1}{2}$. If $G$ is a maximal $C_{2k+1}$-free graph in $\mathcal{G}_{k,\alpha}(n)$ with $n\geq \frac{36k}{\alpha}$, then $e(G) \geq \frac{n^2}{4}-2\sqrt{k\alpha} n^{\frac{3}{2}}$ and any induced complete bipartite subgraph of $G$ has at most $(1-\frac{\alpha}{4})n$ vertices.
\end{prop}

\begin{proof}
By Definition \ref{defn-1} (iii) and (iv), $G[X\cup Y]$ is a bipartite graph with vertex classes $X$ and $Y$. Since $G[X_i,Y_i]$ is empty, we have
\begin{align}\label{eq-4-1}
e(G[X\cup Y]) &= |X||Y|- \sum_{i=1}^t |X_i||Y_i|\nonumber\\
&\geq \left\lfloor\frac{ n-(2k-1)t}{2}\right\rfloor^2 - t \left\lfloor \frac{\alpha n-(2k-1)t}{2t}\right\rfloor^2\nonumber\\
&\geq \left(\frac{n-(2k-1)t-t}{2}\right)^2- t \left(\frac{\alpha n-(2k-1)t}{2t}\right)^2\nonumber\\
&\geq \frac{n^2}{4} +k^2t^2-ktn- \frac{\alpha^2 n^2}{4t}- k^2t +(k-1)\alpha n.
\end{align}
Substituting $t =\lceil \sqrt{\frac{\alpha n}{4k}} \rceil$ into \eqref{eq-4-1}, we arrive at

\begin{align*}
e(G[X\cup Y]) &\geq \frac{n^2}{4} +k^2\frac{\alpha n}{4k}-kn\left(\sqrt{\frac{\alpha n}{4k}}+1\right)- \frac{\alpha^2n^2}{4\sqrt{\frac{\alpha n}{4k}}} - k^2\left(\sqrt{\frac{\alpha n}{4k}}+1\right) +(k-1)\alpha n\\
&= \frac{n^2}{4}-\left(\frac{\sqrt{\alpha k}}{2}+\frac{\alpha^{\frac{3}{2}}\sqrt{k}}{2}\right)n^{\frac{3}{2}}-\left(k-\frac{(5k-4)\alpha}{4}\right)n-\frac{k^{\frac{3}{2}}\sqrt{\alpha n}}{2}-k^2.
\end{align*}
Since $\alpha <1$ and $n\geq \frac{4k}{\alpha}$, we further obtain that
\begin{align*}
e(G[X\cup Y]) &\geq \frac{n^2}{4}-\sqrt{ k\alpha}n^{\frac{3}{2}}-kn-\frac{k^{\frac{3}{2}}\sqrt{\alpha n}}{2}+\left(\frac{(5k-4)\alpha}{4}n-k^2\right)\\
&\geq \frac{n^2}{4} - \sqrt{ k\alpha} n^{\frac{3}{2}}-2kn+\left(kn-\frac{k^{\frac{3}{2}}\sqrt{\alpha n}}{2}\right)\\
&\geq \frac{n^2}{4} - 2\sqrt{ k\alpha} n^{\frac{3}{2}}+\left(\sqrt{ k\alpha} n^{\frac{3}{2}}-2kn\right)\\
&\geq \frac{n^2}{4} - 2\sqrt{ k\alpha} n^{\frac{3}{2}}.\\
\end{align*}

In the following, we shall show that any induced complete bipartite subgraph of $G$ has at most $(1-\frac{\alpha}{4})n$ vertices. Assume that $H$ is a largest induced complete bipartite subgraph of $G$ with vertex classes $A$ and $B$. Note that each vertex of $X_i$ (or $Y_i$)
plays the same role in $G$. If there is a vertex in $X_i$ (or $Y_i$)  belongs to $V(H)$, then by the maximality of $H$, every vertex of $X_i$ (or $Y_i$) belongs to $V(H)$.

Suppose first that $X_{t+1}\cap( A\cup B)=\emptyset$ or $Y_{t+1}\cap( A\cup B)=\emptyset$, then
\begin{align}\label{ineq-4-2}
|A|+|B| &\leq n-\min \{|X_{t+1}|,|Y_{t+1}|\}\nonumber\\
&\leq n -\left(\frac{|X_{t+1}|+|Y_{t+1}|}{2}-1\right)\nonumber\\
&= n -\frac{1}{2}\left(|X_{t+1}|+|Y_{t+1}|\right)+1.
\end{align}
By  Definition \ref{defn-1} (i) and (ii),
\begin{align}\label{ineq-4-3}
|X_{t+1}|+|Y_{t+1}|&=n-\sum_{i=1}^t(|X_i|+|Y_i|)-|Z|\nonumber \\
 &= n-2t \left \lfloor \frac{\alpha n-(2k-1)t}{2t} \right\rfloor -(2k-1)t\nonumber\\
 &\geq (1-\alpha) n.
\end{align}
Combining \eqref{ineq-4-2} and \eqref{ineq-4-3}, we obtain that
\begin{align*}
|A|+|B| \leq n -\frac{(1-\alpha) n}{2} +1\leq \frac{3n}{4} +1\leq (1-\frac{\alpha}{4})n.
\end{align*}

Now consider that $X_{t+1}\cap( A\cup B)\neq \emptyset$ and $Y_{t+1}\cap( A\cup B)\neq\emptyset$, then $X_{t+1},Y_{t+1}\subset A\cup B$. Without loss of generality, we may assume that $X_{t+1}\subset A$ and $Y_{t+1}\subset B$.
Since $G[X_i,Y_i]$ is empty and both $G[X_i,Y_{t+1}]$ and $G[X_{t+1},Y_i]$ are complete bipartite, it follows that at most one of $X_i$ and $Y_i$ is in $A\cup B$. Thus
\begin{align*}
|(A\cup B)\cap (X\cup Y)|&\leq |X_{t+1}|+|Y_{t+1}| +t \left\lfloor\frac{\alpha n-(2k-1)t}{2t}\right\rfloor\\
&=n-t\left\lfloor \frac{\alpha n-(2k-1)t}{2t} \right\rfloor-(2k-1)t\\
&\leq \left(1-\frac{\alpha}{2}\right)n-\frac{2k-3}{2}t.
\end{align*}
Then
\begin{align}\label{eq-4-4}
|A\cup B| &\leq |(A\cup B)\cap (X\cup Y)|+|Z| \nonumber\\
&\leq \left(1-\frac{\alpha}{2}\right)n-\frac{2k-3}{2}t+(2k-1)t\nonumber\\
& = \left(1-\frac{\alpha}{2}\right)n+\frac{2k+1}{2}t.
\end{align}
Substituting $t =\lceil \sqrt{\frac{\alpha n}{4k}} \rceil$ into \eqref{eq-4-4} and  since $n\geq \frac{36k}{\alpha}$, we get
\begin{align*}
|A\cup B|& \leq \left(1-\frac{\alpha}{2}\right)n+\frac{2k+1}{2}\left( \sqrt{\frac{\alpha n}{4k}}+1\right)\\
&\leq \left(1-\frac{\alpha}{2}\right)n+(2k+1)\sqrt{\frac{\alpha n}{4k}} \\
&\leq \left(1-\frac{\alpha}{4}\right)n.\\
\end{align*}
This completes the proof.
\end{proof}

\noindent{\bf Remark.} Very recently, Gerbner \cite{gerbner} extends our result to 3-chromatic graphs with a color-critical edge. They also proposed two interesting conjectures on this topic.


\end{document}